\documentclass[12pt]{amsart}
\newtheorem{theorem}{Theorem}[section]
\newtheorem{lemma}[theorem]{Lemma}
\newtheorem{corollary}[theorem]{Corollary}
\newtheorem{proposition}[theorem]{Proposition}

\newtheorem{remark}[theorem]{Remark}

\numberwithin{equation}{section}

\begin{document}
\baselineskip=15.5pt

\title[Brauer group of moduli of principal bundles]{Brauer
group of moduli of principal bundles over a curve}

\author[I. Biswas]{Indranil Biswas}

\address{School of Mathematics, Tata Institute of Fundamental
Research, Homi Bhabha Road, Bombay 400005, India}

\email{indranil@math.tifr.res.in}

\author[Y. I. Holla]{Yogish I. Holla}

\address{School of Mathematics, Tata Institute of Fundamental
Research, Homi Bhabha Road, Bombay 400005, India}

\email{yogi@math.tifr.res.in}



\date{}

\begin{abstract}
Let $G$ be a semisimple linear algebraic group defined over
the field $\mathbb C$, and let $C$ be an irreducible
smooth complex projective curve of genus
at least three. We compute the Brauer group of the smooth locus
of the moduli space of semistable principal $G$--bundles over $C$.
We also compute the Brauer group of the moduli stack
of principal $G$--bundles over $C$.
\end{abstract}

\maketitle

\section{Introduction}\label{sec1}

The base field will always be $\mathbb C$.
Let $Y$ be a smooth quasiprojective variety, or more generally, a
smooth algebraic stack. Using the isomorphism ${\mathbb C}^r\otimes
{\mathbb C}^{r'}\,=\,{\mathbb C}^{rr'}$,
we have a homomorphism
$\text{PGL}(r,{\mathbb C})\times \text{PGL}(r',{\mathbb C})
\,\longrightarrow\, \text{PGL}(rr',{\mathbb C})$. So a
principal $\text{PGL}(r,{\mathbb C})$--bundle $\mathbb P$ and
a principal $\text{PGL}(r',{\mathbb C})$--bundle
${\mathbb P}'$ on $Y$ together produce
a principal $\text{PGL}(rr',{\mathbb C})$--bundle
on $Y$, which we will denote by ${\mathbb P}\otimes {\mathbb P}'$.
The two principal
bundles $\mathbb P$ and ${\mathbb P}'$ 
are called \textit{equivalent} if there are vector bundles
$V$ and $V'$ on $Y$ such that the principal bundle
${\mathbb P}\otimes {\mathbb P}(V)$ is isomorphic to
${\mathbb P}'\otimes {\mathbb P}(V')$. The equivalence classes
form a group which is called the \textit{Brauer group} of $Y$.
The addition operation is defined
by the tensor product, and the inverse is defined to be the dual
projective bundle. The Brauer group of $Y$ will be
denoted by $\text{Br}(Y)$.

Let $C$ be an irreducible smooth complex projective curve such that
$\text{genus}(C)\,\geq\, 3$. Fix a line bundle $L$ over $C$
of degree $d$. Let $M(r,L)$ be the moduli space of stable vector
bundles $E$ over $C$ of rank $r$ with $\bigwedge^r E\,=\, L$.
The Brauer group $\text{Br}(M(r,L))$ is cyclic of order
$\text{g.c.d.}(r\, ,\text{degree}(L))$ \cite{BBGN}. A generator
for $\text{Br}(M(r,L))$ is obtained by restricting the
universal projective bundle over $C\times M(r,L)$ to $\{x_0\}
\times M(r,L)$, where $x_0\,\in\, C$ is some fixed point.
The Brauer group of the smooth locus of
stable principal $\text{PGL}(r,{\mathbb C})$--bundles over $C$
was computed in \cite{BHO}.

Let $G$ be a semisimple linear algebraic group. Let
${\mathcal M}_C(G)$ be the moduli stack of principal $G$--bundles on
$C$. The coarse moduli space of semistable principal $G$--bundles on
$C$ will be denoted by $M^{\rm ss}_C(G)$. Let
$$
M_C(G)^{\rm rs}\,\subset\, M^{\rm ss}_C(G)
$$
be the smooth locus. It is known that $M_C(G)^{\rm rs}$ coincides with
the locus of regularly stable principal $G$--bundles on
$C$ (a stable principal $G$--bundle $E_G$ is called regularly stable
if the natural homomorphism from the center of $G$ to the
group of automorphisms of $E_G$ is surjective).
Our aim here is to compute the Brauer groups of $M_C(G)^{\rm rs}$
and ${\mathcal M}_C(G)$.

We prove the following theorem (see Theorem \ref{thm0} and
Corollary \ref{sc}):

\begin{theorem}\label{th-i1}
Assume that $G$ is semisimple and simply connected.
Then the Brauer group
of ${\mathcal M}_C(G)$ is trivial, and the Brauer group
${\rm Br}(M_C(G)^{\rm rs})$ is the group of characters
$Z^{\vee}_G$ of the center $Z_G\, \subset\, G$.
\end{theorem}

If $G$ is not simply connected, then the Brauer group of the moduli
stack ${\mathcal M}_C(G)$ is computed in Theorem \ref{direct}
(in conjunction with Proposition \ref{LCG-2}). The Brauer group
of $M_C(G)^{\rm rs}$ is computed in Theorem \ref{th-m-s}.

We will describe a corollary of these computations.

Let $Z_{\widetilde G}$ be the center of the universal
cover $\widetilde G$ of $G$. Define
$$
\Psi(G)\,\subset\, 
{\rm Hom}(Z_{\widetilde G}\otimes_{\mathbb Z} 
Z_{\widetilde G}\, , {\mathbb Q}/{\mathbb Z})
$$
as in \eqref{PsiG}. Given an element $\delta \,\in \,\pi_1(G)
\, \subset\, Z_{\widetilde G}$,
define
$$
{\rm ev}_G^{\delta}\,:\, \Psi(G) \,\longrightarrow\,
{\rm Hom}(Z_{\widetilde G}/{\pi_1(G)}\, , {\mathbb Q}/{\mathbb Z})
$$
as in \eqref{evG}. The connected components of both
${\mathcal M}_C(G)$ and $M^{\rm ss}_C(G)$ are parametrized
by $\pi_1(G)$. For any $\delta\, \in \, \pi_1(G)$, let
${\mathcal M}_C(G)^\delta\, \subset\
{\mathcal M}_C(G)$ be the
connected component corresponding to $\delta$.
Let $M_C(G)^{\delta,{\rm rs}}$ be the smooth locus
of the connected component, corresponding to $\delta$, of
the moduli space $M^{\rm ss}_C(G)$.

We prove the following (see Corollary \ref{corl}):

\begin{theorem}\label{th-i2}
There is a short exact sequence
$$
0\,\longrightarrow \,{\rm Coker}({\rm
ev}^\delta_G)\,\longrightarrow\,
{\rm Br}(M_C(G)^{\delta,{\rm rs}}) \,\longrightarrow
\,{\rm Br}({\mathcal M}_C(G)^\delta)\,\longrightarrow \,0\, .
$$
\end{theorem}

The Brauer groups for the moduli spaces of principal bundles with
a classical group as the structure group are computed in Section
\ref{sec7}. As a consequence, the twisted moduli space of ${\rm
Sp}_{2n} ({\mathbb C})$--bundles is locally factorial if $n$ is
odd, and it is not locally factorial $n$ is even. (See
Corollary \ref{cor1}.)

We note that among the exceptional groups, $G_2$, $F_4$ and $E_8$
have the property that the center is trivial. Therefore, from
Theorem \ref{th-i1} it follows that for these three groups
$$
{\rm Br}({\mathcal M}_C(G))\,=\,0\,=\,{\rm Br}(M_C(G)^{\rm rs})\, .
$$

\medskip
\noindent
\textbf{Acknowledgements.}\, We are grateful to the referee for
providing helpful comments.

\section{Preliminaries}\label{sec2}

As mentioned before, the base field
is $\mathbb C$. For an algebraic
stack $Y$, the torsion group $H^2_{{\text {\'et}}}(Y,\, {\mathbb 
G}_m)_{{\rm torsion}}$ is called the \textit{cohomological
Brauer group} of $Y$. A theorem of Grothendieck says that
if $Y$ is a smooth variety, then $H^2_{{\text {\'et}}}(Y,\, {\mathbb 
G}_m)$ is torsion \cite{Mi}. If $Y$ is a quasiprojective variety,
then a theorem of Gabber says that
$H^2_{{\text {\'et}}}(Y,\, {\mathbb G}_m)_{{\rm torsion}}$
coincides with the Brauer group defined by the Morita equivalence
classes of Azumaya algebras over $Y$; see \cite[Theorem 1.1]{dJ}
for this theorem of Gabber.

As mentioned in the introduction,
we are interested in computing the Brauer groups of the moduli stack of
principal $G$--bundles as well as the smooth
locus of the moduli space of semistable principal $G$--bundles over a 
smooth projective curve. Since the smooth locus of the moduli space is 
a smooth quasiprojective variety, the above theorems of Grothendieck
and Gabber imply that all the above
three groups associated to it coincide.

For algebraic stacks, we will only consider the cohomological
Brauer group. It should be mentioned that the moduli stack of
principal $G$--bundles over a smooth projective curve is smooth.

\begin{proposition} \label{key}
Let $Y$ be an algebraic stack satisfying the following two
properties: 
\begin{itemize}
\item any class in $H^2(Y,\, {\mathbb Z})$ is represented by a
holomorphic line bundle on $Y$, and

\item each holomorphic line bundle on $Y$ admits an algebraic
structure.
\end{itemize}
Then there are isomorphisms
$$
H^2_{{\text {\'et}}}(Y,{\mathbb G}_m)_{{\rm torsion}}\,\cong\,
H^2(Y,\, {\mathcal O}_{X,{\rm an}}^*)_{{\rm torsion}} \,\cong\,
H_B^3(Y,\, {\mathbb Z})_{{\rm torsion}}\, .
$$
\end{proposition}

\begin{proof}
For finite coefficients, we have a comparison isomorphism between
the \'etale cohomology and the Betti cohomology.
Let $\mu_n\,\subset\, {\mathbb C}^*$ be the group of $n$--th
roots of $1$.
By the Kummer sequence, any element in $H^2_{\text{\'et}}(Y,\,{\mathbb
G}_m)_{{\rm torsion}}$ (respectively, 
$H^2_{\rm an}(Y,\, {\mathcal O}_{Y,{\rm an}}^*)_{{\rm torsion}}$) is 
represented
by a class in $H^2_{{\text {\'et}}}(Y,\, {\mu}_n)$ 
(respectively, $H^2(Y,\,{\mathbb Z}/n{\mathbb Z})$) for
some $n$. This implies the surjectivity of the homomorphism
$$
H^2_{{\text {\'et}}}(Y,\, {\mathbb G}_m)_{{\rm torsion}}\,
\longrightarrow\, H^2(Y,\,{\mathcal O}_{Y,{\rm an}}^*)_{{\rm
torsion}}\, ,
$$
while its injectivity follows from the assumptions by a
straight--forward diagram chase.

The second isomorphism in the proposition is
derived using the exact sequence
$$
0\, \longrightarrow\,{\mathbb Z}\, \longrightarrow\,
{\mathcal O}_{Y,{\rm an}}\,\longrightarrow\,
{\mathcal O}_{Y,{\rm an}}^*\,\longrightarrow\, 0
$$
associated to the homomorphism $f\,\longmapsto\,
\exp(2\pi\sqrt{-1}f)$.
\end{proof}

See \cite[p. 25, Theorem 1]{Ne} for a similar result.

\section{The ind--Grassmannian and $L_C(G)$}

Let $G$ be a connected semisimple linear algebraic group over
$\mathbb C$. Let
$$
{\widetilde G}\,\longrightarrow\, G
$$
be the universal cover. The fundamental
group of $G$ will be denoted by $\pi_1(G)$; being abelian
$\pi_1(G)$ is independent of the choice of base point in $G$.
Let
$$
L(G)\,:=\,G((t)) ~\, ~\,\text{~and~}~\, ~\, L^+(G)\,:=\,G[[t]]
$$
be the loop group and its naturally defined subgroup scheme
respectively. Let 
$$
{\mathcal Q}_G\,:=\, L(G)/L^+(G)
$$
be the ind--Grassmannian; it is a direct limit of projective integral
varieties.

Let $C$ be an irreducible smooth complex projective curve with
$\text{genus}(C)\, \geq\, 3$. Fix a base
point $p_0\, \in\, C$. Let
\begin{equation}\label{llc}
L_C(G)\, :=\, G({\mathcal O}(C-p_0))\, \subset\, L(G)
\end{equation}
be the sub--ind--group scheme.

Let ${\mathcal M}_C(G)$ be the moduli stack of principal $G$--bundles on 
$C$.
The ``Uniformization theorem'' produces a canonical isomorphism
$$
L_C(G)\backslash L(G)/L^+(G)\,=\, L_C(G)\backslash
{\mathcal Q}_G\, \stackrel{\sim}{\longrightarrow}\,{\mathcal M}_C(G)
$$
\cite{F}, \cite{Te}, \cite{BLS}.

\begin{proposition}\label{ind-gr}
The cohomological Brauer group
$H^2_{{\text {\'et}}}({\mathcal Q}_G,\, {\mathcal
O}^*_{{\mathcal Q}_G})_{{\rm torsion}}$ of
${\mathcal Q}_G$ is trivial.
\end{proposition}

\begin{proof}
We begin by recalling a lemma from \cite{BLS} (see
\cite[p. 185, Lemma 1.2]{BLS}).

\begin{lemma}[\cite{BLS}]\label{lem00}
The quotient map $L(G)\,\longrightarrow\,{\mathcal Q}_G$ induces a 
bijection 
$$
\pi_1(G)\,=\, \pi_0(L(G))\,\longrightarrow \,\pi_0({\mathcal Q}_G)\, .
$$
Each connected component of ${\mathcal Q}_G$ is isomorphic to
${\mathcal Q}_{{\widetilde G}}$, where $\widetilde G$ is the
universal cover of $G$.
\end{lemma}

In view of Lemma \ref{lem00}, it is enough to prove the
proposition for simply connected groups.

First assume that the group $G$ is almost simple and simply connected.
 
Fix a Borel subgroup $B\, \subset\, G$ together with a maximal torus 
$T\,\subset\,B$ of $G$. The
homomorphism of rings ${\mathbb 
C}[[t]]\,\longrightarrow\, {\mathbb C}$ defined
by $t\, \longmapsto\, 0$ induces a homomorphism 
$L^+(G)\,\longrightarrow\, G$. Let
$$
{\mathcal B}\,\subset\, L^+(G)
$$
be the inverse image of the Borel subgroup
$B$ under this homomorphism. We recall that
${\mathcal B}$ is called the standard Borel subgroup of $L(G)$
associated to $B$. Let $N(T)$ be the normalizer of $T$ in $G$.
Let
$$
{\widetilde W}\,:=\, {\rm Mor}({\mathbb C}^*, N(T))/T
$$
be the affine Weyl group containing the Weyl group $W\,=\,N(T)/T$
as a subgroup (the constant morphisms from ${\mathbb C}^*$
to $N(T)$ make $N(T)$ a subgroup of
${\rm Mor}({\mathbb C}^*, N(T))$).

The ind--Grassmannian ${\mathcal Q}_G$ has the Bruhat decomposition
$$
{\mathcal Q}_G\,=\, \bigcup_{w\in {\widetilde W}/W}{\mathcal B}w'L^+(G)/
L^+(G)\, ,
$$
where $w'\in {\widetilde W}$ is any chosen representative of $w$.
The quotient space ${\widetilde W}/W$ has a partial ordering
inherited from the Bruhat partial ordering of the affine Weyl group 
${\widetilde W}$.
This makes ${\mathcal Q}_G$ a direct limit of projective
varieties $\{Q_w\}_{w\in {\widetilde W}/W}$, where
\begin{equation}\label{gsv}
Q_w\,:=\, \bigcup_{v \le w}{\mathcal B}vL^+(G)/L^+(G)\, .
\end{equation}
For any $w \,\in\, {\widetilde W}/W$, consider the length of
every element in the coset $w$. The smallest one among them
will be denoted by $\ell (w)$.

The generalized Schubert variety 
$Q_w$ in \eqref{gsv} has a Zariski
open subset defined by ${\mathcal B}wL^+(G)/L^+(G)$ which is
biregularly isomorphic with the affine space ${\mathbb A}^{\ell (w)}$,
where $\ell (w)$ is defined above.
As a result, $Q_w$ equipped with the analytic topology
has the structure of a CW complex with only even dimensional cells.
Hence $H^3(Q_w,\,{\mathbb Z})\,=\,0$.
This implies that
$$
H^3({\mathcal Q}_G,\, {\mathbb Z})\,=\, 0
$$
because ${\mathcal Q}_G$ is a direct limit of these varieties $Q_w$.
In particular,
\begin{equation}\label{e-Q-v}
H^3(Q_w,\,{\mathbb Z})_{\rm torsion}\,=\,0\,=\,
H^3({\mathcal Q}_G,\,{\mathbb Z})_{\rm torsion}\, .
\end{equation}
Similarly, we have
\begin{equation}\label{e1}
H^1({\mathcal Q}_G,\,{\mathbb Z})\,=\, 0\, .
\end{equation}

Both $Q_w$ and ${\mathcal Q}_G$ satisfy the two assumptions 
in Proposition \ref{key} \cite[p. 157, Lemma 2.2]{KN}. 
Therefore, the proposition
follows from Proposition \ref{key} and \eqref{e-Q-v} under the
assumption that $G$ is almost simple and simply connected.

In the general case where $G$ is semisimple and simply connected,
write $G$ as a product
$$
G\,=\, \prod_{i=1}^s G_i\, ,
$$
where each $G_i$ is almost simple and simply connected.
This enables us to write the ind--Grassmannian
${\mathcal Q}_G$ as a product
$$
{\mathcal Q}_G\,=\, \prod_{i=1}^s {\mathcal Q}_{G_i}\, ,
$$
where ${\mathcal Q}_{G_i}$ is the ind--Grassmannian for $G_i$.
In view of \eqref{e-Q-v} and \eqref{e1} for $G_i$, from the
K\"unneth decomposition of $H^3(\prod_{i=1}^s {\mathcal Q}_{G_i},\, 
{\mathbb Z})$ we conclude that $H^3({\mathcal Q}_G,\,{\mathbb
Z})_{\rm torsion}\,=\, 0$. The two conditions in Proposition
\ref{key} hold for ${\mathcal Q}_G$ \cite[p. 157, Lemma 2.2]{KN}.
Now the proof of the proposition is completed using Proposition 
\ref{key}.
\end{proof}

\subsection{Some properties of $L_C(G)$}
Let $G$ be semisimple and simply connected. The topological
properties of the ind--group scheme $L_C(G)$ that we need
can be derived from the following theorem of
Teleman (\cite[p. 8, Theorem 1]{Te}):

\begin{theorem}[\cite{Te}]\label{thm-T}
The natural map $L_C(G)\,\longrightarrow\, C^{\infty}
(C\setminus\{p_0\},G)$ defines a homotopy equivalence. Hence
the homotopy type of $L_C(G)$, equipped with the analytic
topology, is that of $G\times\Omega G^{2g}$.
\end{theorem}

The connectedness and simply connectedness of
$L_C(G)$ follow immediately from Theorem \ref{thm-T}.

\begin{proposition}\label{LCG}
Let $G$ be semisimple and simply connected. Let $BL_C(G)$ be the
classifying space for $L_C(G)$. Then
\begin{enumerate}
\item $H^1(BL_C(G),\, {\mathbb Z})\,=\,0$,
\item $H^2(BL_C(G),\, {\mathbb Z}/n{\mathbb Z})\,=\,0$
for all $n$, and
\item $H^2(BL_C(G),\, {\mathbb C}^*)\,=\,0$. 
\end{enumerate}
\end{proposition}

\begin{proof}The group $H^1(BL_C(G),\, {\mathbb Z})$ parametrizes the 
space of all
continuous homomorphisms from $L_C(G)$ to ${\mathbb Z}$. So
$H^1(BL_C(G),\, {\mathbb Z})$ is trivial by the connectedness of $L_C(G)$.

The connectedness of $L_C(G)$ also implies that $\pi_1(BL_C(G))\,=\,
\pi_0(L_C(G))\,=\,0$. The simply connectedness of $L_C(G)$ implies
that $\pi_2(BL_C(G))\,=\,\pi_1 (L_C(G))\,=\,0$. Hence
statements (2) and (3)
of the proposition follow from the Hurewicz's theorem.
\end{proof}

\section{Brauer group of moduli: $G$ is simply 
connected}\label{sec4}

\begin{theorem} \label{thm0}
Let $G$ be simply connected and semisimple. Then
${\rm Br}({\mathcal M}_C(G))\,=\,0$.
\end{theorem}

\begin{proof}
The descent spectral sequence for the principal $L_C(G)$--bundle 
${\mathcal Q}_G \,\longrightarrow \,{\mathcal M}_C(G)$ gives the 
following exact sequence in the analytic topology:
\begin{equation}\label{Sp-S}
0 \,\longrightarrow\,
H^1(BL_C(G), \,{\mathbb C}^*) \,\longrightarrow\,H^1({\mathcal
M}_C(G), \, {\mathcal O}^*)\,\stackrel{\theta}{\longrightarrow}
\, H^0(BL_C(G),\, H^1({\mathcal Q}_G,{\mathcal O}^*))
\end{equation}
$$
\longrightarrow\, H^2(BL_C(G), \,{\mathbb C}^*) \,\longrightarrow\,
{\rm kernel}[H^2({\mathcal M}_C(G),\, {\mathcal O}^*)\,\longrightarrow
\,H^0(BL_C(G),\, H^2({\mathcal Q}_G,{\mathcal O}^*))]
$$
$$
\longrightarrow \, H^1(BL_C(G),\, H^1({\mathcal Q}_G,{\mathcal 
O}^*))\, ;
$$
cf. \cite[p. 371, Corollary 3.2]{Me} (see also \cite[p. 10, 
(1.9)]{Te}, \cite[p. 27, (5.5)]{Te}).

Let $s$ be the number of almost simple factors in the product 
decomposition of $G$.
We have $H^2_{\text {\'et}}({\mathcal Q}_G,\, {\mathcal
O}^*_{{\mathcal Q}_G})_{{\rm torsion}}\,=\,0$ by
Proposition \ref{ind-gr}. Also, $H^1({\mathcal Q}_G,\, {\mathcal 
O}^*)\,\cong\, {\mathbb Z}^s$ \cite[p. 186, Lemma 1.4]{BLS}.
Therefore, from Proposition \ref{LCG}(1),
\begin{equation}\label{e2}
H^0(BL_C(G),\, H^2({\mathcal Q}_G,{\mathcal O}^*_{{\mathcal Q}_G}))
\,=\, 0\,=\,
H^1(BL_C(G), \, H^1({\mathcal Q}_G,{\mathcal O}^*_{{\mathcal Q}_G}))\, .
\end{equation}

The homomorphism of the Picard groups 
\begin{equation}\label{Pic}
\theta \, :\, {\rm Pic}({\mathcal M}_C(G))\,\longrightarrow \,
{\rm Pic}({\mathcal Q}_G)\,\cong\,{\mathbb Z}^s
\end{equation}
in \eqref{Sp-S}
has a finite cokernel \cite[p. 187, Proposition 1.5]{BLS}, 
in particular, this cokernel is a
torsion group. Since $H^2(BL_C(G),\, {\mathbb C}^*)\,=\, 0$ 
(see Proposition \ref{LCG}(3)), from \eqref{Sp-S}
it follows that $\theta$ in \eqref{Pic} is actually surjective.
Therefore, using \eqref{e2}, from \eqref{Sp-S} we conclude that
$$
H^2({\mathcal M}_C(G),\, {\mathcal O}^*)\, =\,
H^2(BL_C(G),\, {\mathbb C}^*)\,=\,0 .
$$
Now we apply Proposition \ref{key} to the stack
${\mathcal M}_C(G)$; from \cite[p. 26, Proposition 5.1]{Te}
and \cite[p. 26, Remark 5.2]{Te} it follows that the
two conditions in Proposition \ref{key} are satisfied.
Therefore, we conclude that ${\rm Br}({\mathcal M}_C(G))
\,=\,0$.
\end{proof}

\begin{remark}
{\rm It should be pointed out that the above arguments involving
sheaves in
analytic topology can be replaced by an argument which uses only
constant sheaves in Euclidean topology.
Using the exponential sequence, the Brauer group ${\rm 
Br}(X)$ for any $X$ can be expressed as the quotient of
$H^2(X,\,{\mathbb C}^*)$ by the image of $H^2(X,\,{\mathbb C})$.
The isomorphism
$$
H^2({\mathcal M}_C(G),\,{\mathbb C})\cong H^2({\mathcal Q}(G),\,{\mathbb 
C})
$$
simplifies the terms in the descent spectral sequence associated to the
constant sheaf ${\mathbb C}^*$.}
\end{remark}

Let
$$
Z_G\, \subset\, G
$$
be the center. Let $M^{\rm ss}_C(G)$
be the coarse moduli space of semistable 
principal $G$--bundles over $C$. The smooth locus of
$M^{\rm ss}_C(G)$ is the locus of regularly stable
principal $G$-bundles over $C$ \cite[Corollary 3.6]{BH2}.
We recall that a principal
$G$--bundle $E$ is called \textit{regularly stable} if
$E$ is stable, and the natural homomorphism
$Z_G\, \longrightarrow\, {\rm Aut(E)}$ is surjective (the
restriction to $Z_G$ of the action of $G$ on $E$ produces this
homomorphism). Let
$$
{\mathcal M}_C(G)^{\rm rs}\, \subset\, {\mathcal M}_C(G)
$$
be the open sub--stack defined by the regularly stable bundles
\cite[Lemma 2.3]{BH}. Let
\begin{equation}\label{eq-c-m}
p\, :\, {\mathcal M}_C(G)^{\rm rs}\, \longrightarrow\,
M_C(G)^{\rm rs}
\end{equation}
be the morphism to the coarse moduli space, so
$M_C(G)^{\rm rs}$ is the moduli space of regularly stable
principal $G$--bundles over $C$. As noted above, $M_C(G)^{\rm rs}$
is the smooth locus of $M^{\rm ss}_C(G)$.
 
\begin{proposition} \label{bis-hoff}
As before, $G$ is semisimple.
\begin{enumerate}
\item ${\mathcal M}_C(G)^{\rm rs}$ is an open sub--stack of the
moduli stack ${\mathcal M}_C(G)$, and the complement of it
in ${\mathcal M}_C(G)$ is of codimension at least two.

\item The morphism $p$ in \eqref{eq-c-m}
defines a gerbe over $M_C(G)^{\rm rs}$ banded by $Z_G$.
\end{enumerate}
\end{proposition}

\begin{proof}
See \cite[Theorem II.6]{F} or Theorem 2.4 of \cite{BH} 
for a proof of the first part. The second part is proved in
Section 6 of \cite{BH}.
\end{proof}

Let
\begin{equation}\label{psi}
\psi\,\in\, H^2(M_C(G)^{\rm rs}, \,Z_G)
\end{equation}
be the cohomology
class defined by the gerbe in \eqref{eq-c-m}.

Let $Z$ be a finite abelian group, and let
\begin{equation}\label{gp}
\alpha\, :\, {\mathcal M}\,\longrightarrow\, M
\end{equation}
be a gerbe
banded by $Z$, where $M$ is irreducible. Take any
line bundle $L$ over ${\mathcal M}$.
So $L$ is given by a functor $L_S$ from ${\mathcal M}(S)$ to 
the groupoid of line bundles on $S$ for every
$\mathbb C$--scheme $S$. In particular, $L_S$ 
defines for every object $\mathcal E$ in ${\mathcal M}(S)$ a group 
homomorphism
$$
L_{S, {\mathcal E}}\,:\, {\rm Aut}_{{\mathcal M}( S)}( \mathcal E) 
\,\longrightarrow\, 
{\rm Aut}_{{\mathcal O}_S}(L_S( 
{\mathcal E})) \,= \,\Gamma( S, {\mathcal O}^*_S)\, .
$$
The compatibility conditions ensure that the composition
$$
Z( S)\,\stackrel{\iota_{\mathcal E}}{\longrightarrow}\,
\text{Aut}_{{\mathcal M}(S)}(\mathcal E) 
\,\stackrel{L_{S, \mathcal E}}{\longrightarrow}\,
\Gamma( S, {\mathcal O}_S^*)
$$
defines a $1$--morphism $Z \times{\mathcal M} \,\longrightarrow\, {
\mathbb G}_m \times {\mathcal M}$ over ${\mathcal M}$.
As $M$ is connected, and $\text{Hom}(Z, {\mathbb G}_m)$ is discrete,
this $1$--morphism 
is the pullback of some character $\chi\,:\, Z \,\longrightarrow\,
{\mathbb G}_m$. We call $\chi$ the \textit{weight} of $L$. (See
 \cite[Section 6]{BH}.)

\begin{lemma}\label{pic-bra}
Let $\beta \,\in\, H^2(M,\,Z)$ be the class
of the gerbe in \eqref{gp} banded by $Z$.
Then there is an exact sequence 
$$
0 \,\longrightarrow \,{\rm Pic}(M) 
\,\stackrel{\alpha^*}{\longrightarrow} 
\,{\rm Pic}({\mathcal M}) \,\stackrel{{\rm wt}}{\longrightarrow}\, 
{\rm Hom}(Z,\, {\mathbb G}_m)\,\stackrel{{\beta_*}}{\longrightarrow}\, 
{\rm Br}(M)\,\stackrel{\alpha^*}{\longrightarrow}\,{\rm Br}({\mathcal 
M})\, ,
$$
where ${\beta_*}$ takes any homomorphism $\eta\,:\,Z\,
\longrightarrow\, {\mathbb G}_m$ to the image of $\beta$ under the 
homomorphism $ H^2(M,\,Z)\longrightarrow H^2(M,\, {\mathbb G}_m)$ 
induced by $\eta$.
\end{lemma}

\begin{proof}
This is the Leray spectral sequence for ${\mathcal O}^*_{\mathcal M}
\,\longrightarrow\, {\mathcal M}\,\stackrel{\alpha}{\longrightarrow}\, 
M$.
The only points to note are that
$$
\alpha_*{\mathcal O}^*_{\mathcal M}\,=\,{\mathcal O}^*_M
$$
and $R^1\alpha_*{\mathcal O}^*_{\mathcal M}$
is the constant sheaf ${\rm Hom}(Z,\, {\mathbb G}_m)$.
\end{proof}

\begin{theorem} \label{gerbe-sc}
There is an exact sequence
$$
0 \,\longrightarrow \,{\mathbb Z}^s\,\stackrel{p^*}{\longrightarrow} 
\,{\mathbb Z}^s \,\stackrel{{\rm wt}}{\longrightarrow}\, 
{\rm Hom}(Z_G,\,{\mathbb G}_m) \,\stackrel{\psi_*}{\longrightarrow}\,
{\rm Br}(M_C(G)^{\rm rs}) \,\longrightarrow\, 0\, ,
$$
where $s$ is the number of almost simple factors in $G$, $p^*$ is
the homomorphism of Picard groups induced by
the morphism $p$ in \eqref{eq-c-m}, ${\rm wt}$ is
the weight map defined above, and $\psi_*$ is constructed
using the class $\psi$ in \eqref{psi} in the obvious way.
\end{theorem}

\begin{proof}
{}From Theorem \ref{thm0} and Proposition \ref{bis-hoff}(1)
it follows that ${\rm Br}({\mathcal M}_C(G)^{\rm rs})\,=\,0$.
Therefore, Lemma \ref{pic-bra} applied to the
gerbe in \eqref{eq-c-m} produces the exact sequence.
\end{proof}

\section{The twisted case}

In this section the semisimple group $G$ is assumed to be
simply connected but the moduli stack and moduli
space will be twisted (defined below).

Fix a maximal torus $T$ of $G$. As before, the center
of $G$ will be denoted by $Z_G$. Let $\sigma$ be the
rank of $G$, so $\sigma$ is the dimension of $T$.

As before, for any $m$,
let $\mu_m\,\subset\, {\mathbb C}^*$ be the group of all $m$--th
roots of $1$. We fix an isomorphism 
\begin{equation}\label{e-rho}
\rho\,:\, T \,\longrightarrow\, ({\mathbb G}_m)^\sigma~\,~\,
\text{~with~}~\, \rho(Z_G) \,=\, \prod_{i=1}^\sigma \mu_{r_i}\, .
\end{equation}
Using this isomorphism $\rho$, the homomorphism
$$
({\mathbb G}_m)^\sigma \,\longrightarrow\, ({\mathbb G}_m)^\sigma
$$
defined by $\prod_{i=1}^\sigma z_i\,\longmapsto\,
\prod_{i=1}^\sigma (z_i)^{r_i}$ produces an isomorphism
$$
T/Z_G\,=\, ({\mathbb G}_m)^\sigma/\rho(Z_G)
\,\stackrel{\sim}{\longrightarrow}\, T\, .
$$

Let $C(G)$ denote the quotient of
$G \times T$ by $Z_G$ for the diagonal action. The projections
$$
q \,: \,C(G) \,\longrightarrow\, T/Z_G\,=\,T ~\,~\text{~and~}
~\,~ p\,:\,C(G) \,\longrightarrow \,G/Z_G
$$
induce morphisms of stacks 
\begin{equation}\label{m-d}
{\rm det}\,:\,{\mathcal M}_C(C(G))\,\longrightarrow \,
{\mathcal M}_C(T)
\end{equation}
and
\begin{equation}\label{m-pi}
{\mathcal M}_C(C(G))\,\longrightarrow \,{\mathcal M}_C(G/Z_G)
\end{equation}
respectively.

For an element ${\bf d}\,=\,(d_1,\ldots,d_\sigma)\,\in\,
{\mathbb Z}^\sigma$, let 
${\mathcal O}_C({\bf d}p)$ denote the rational point of ${\mathcal
M}_C(T)$ defined by $({\mathcal O}_C(d_1p_0),\ldots,{\mathcal
O}_C(d_\sigma p_0))$.
Define
\begin{equation}\label{m-d2}
{\mathcal M}^{\bf d}_C(G)\, :=\, 
{\rm det}^{-1}({\mathcal O}_C({\bf d}p))\, \subset\,
{\mathcal M}_C(C(G))\, ,
\end{equation}
where ${\rm det}$ is the morphism in \eqref{m-d}. Let
\begin{equation}\label{m-pi2}
\Phi\,:\,
{\mathcal M}^{\bf d}_C(G)\,\longrightarrow \,{\mathcal M}_C(G/Z_G)
\end{equation}
be the restriction of the morphism in \eqref{m-pi}.

Take any $\delta\, \in\, Z_G$. Consider $$\rho(\delta)\,=\,
(a_1\, ,\cdots\, ,a_\sigma)\, \in\, ({\mathbb G}_m)^\sigma\, ,$$
where $\rho$ is the homomorphism in \eqref{e-rho}. Take
${\bf d}\,\in\, {\mathbb Z}^\sigma$ such that there is
an element $${\bf r}\,:=\, (r_1,\ldots,r_\sigma)\, \in\,
{\mathbb Z}^\sigma$$ satisfying the condition that
$\exp(-2\pi\sqrt{-1}d_i/r_i)\,=\, a_i$ for every $i\,\in\,
[1\, ,\sigma]$.

The connected components of ${\mathcal M}_C(G/Z_G)$ are
parametrized by $\pi_1(G/Z_G)\,=\,Z_G$ (see Lemma \ref{lem00});
recall that $G$ is simply connected. Let
$$
{\mathcal M}_C(G/Z_G)^{\delta}\,\subset\,
{\mathcal M}_C(G/Z_G)
$$
be the connected component corresponding to
the element $\delta$. Let
$$
{\mathcal M}^{\delta}_C(G)\, :=\,
\Phi^{-1}({\mathcal M}_C(G/Z_G)^{\delta})\,\subset\,
{\mathcal M}^{\bf d}_C(G)
$$
be the open and closed sub--stack of ${\mathcal M}^{\bf d}_C(G)$,
where $\Phi$ is the morphism in \eqref{m-pi2}.
This ${\mathcal M}^{\delta}_C(G)$ is called the
\textit{twisted} moduli stack (see \cite[Section 2]{BLS}).

The restriction
\begin{equation}\label{phr}
\Phi\vert_{{\mathcal M}^{\delta}_C(G)}
\,:\, {\mathcal M}^{\delta}_C(G)\,\longrightarrow\, {\mathcal
M}_C(G/Z_G)^{\delta}
\end{equation}
is surjective.

The group $G/Z_G$ will also be denoted by $G_{\rm ad}$. The
connected components of $L(G_{\rm ad})$ are parametrized by
$Z_G$ (see Lemma \ref{lem00}). Let
\begin{equation}\label{h1}
(LG_{\rm ad})^{\delta}\, \subset\, L(G_{\rm ad})
\end{equation}
be the connected component corresponding to $\delta$.

The following proposition is proved in \cite{BLS} (see
\cite[p. 189, (2.4)]{BLS}):

\begin{proposition}[\cite{BLS}]\label{unifG0}
For any $\delta\in Z_G$, and any  $\zeta \,\in\,
(LG_{\rm ad})^{\delta}({\mathbb C})$, where $(LG_{\rm ad})^{\delta}$
is constructed in \eqref{h1}, there is a natural isomorphism
$$
{\mathcal M}^{\delta}_C(G)\,\cong\, (\zeta^{-1}L_C(G) \zeta) 
\backslash {\mathcal Q}_{{\widetilde G}}\, .
$$
\end{proposition}

Proposition \ref{unifG0} implies that the of proof of Theorem
\ref{thm0} goes through in the twisted case being considered.
So we have the following theorem:
 
\begin{theorem} \label{thm0-d}
For any $\delta \,\in\, Z_G$, the group 
${\rm Br}({\mathcal M}^{\delta}_C(G))$ is trivial.
\end{theorem}

There is a coarse moduli space $M^{\delta}_C(G)^{\rm ss}$ for 
the open sub--stack of ${\mathcal M}^{\delta}_C(G)$
defined by the locus of semistable principal $C(G)$--bundles.
There is also an open subscheme
$M^{\delta}_C(G)^{\rm rs}\,\subset\, M^{\delta}_C(G)^{\rm ss}$ 
corresponding to the regularly stable principal
$C(G)$--bundles.

\begin{proposition} \label{bis-hoff-twisted}
The codimension of the complement of ${\mathcal M}^{\delta}_C(G)^{\rm 
rs}$ in ${\mathcal M}^{\delta}_C(G)$ is at least two.

The smooth locus of $M^{\delta}_C(G)^{\rm ss}$ is 
$M^{\delta}_C(G)^{\rm rs}$.
\end{proposition}

\begin{proof} 
Let ${\mathcal M}^{\delta}_C(G/Z_G)^{\rm rs}$ be the sub--stack of
${\mathcal M}^{\delta}_C(G/Z_G)$ corresponding to the regularly stable
principal $C(G)$--bundles; that it is a sub--stack follows from
\cite[Lemma 2.3]{BH}. By \cite[Theorem II.6]{F} or \cite[Theorem 
2.4]{BH} it follows
that the complement of ${\mathcal M}^{\delta}_C(G/Z_G)^{\rm rs}$ in 
${\mathcal M}^{\delta}_C(G/Z_G)$ has codimension at least two. Taking
inverse image for the morphism $\Phi\vert_{{\mathcal 
M}^{\delta}_C(G)}$ in \eqref{phr}, the 
first part of the proposition follows.

The proof that the smooth locus of $M^{\delta}_C(G)^{\rm ss}$ is 
$M^{\delta}_C(G)^{\rm rs}$ is identical to that of
\cite[Corollary 3.6]{BH2}.
\end{proof}

The morphism to the coarse moduli space
\begin{equation}\label{p11}
p\, :\, {\mathcal M}^{\delta}_C(G)^{\rm rs} \,\longrightarrow\,
M^{\delta}_C(G)^{\rm rs}
\end{equation}
defines a gerbe banded by $Z_G$. Let
\begin{equation}\label{psi11}
\psi\,\in\, H^2(M^{\delta}_C(G)^{\rm rs},\, Z_G)
\end{equation}
be the class of this gerbe.

Just like Theorem \ref{gerbe-sc},
we now have the following moduli space version of Theorem 
\ref{thm0-d}.

\begin{theorem} \label{exact-gerbe}
There is an exact sequence
$$
0 \longrightarrow {\rm Pic}(M^{\delta}_C(G)^{\rm rs})
\stackrel{p^*}{\longrightarrow}
{\rm Pic}({\mathcal M}^{\delta}_C(G)^{\rm rs})
\stackrel{{\rm wt}}{\longrightarrow} 
{\rm Hom}(Z_G,{\mathbb G}_m)\stackrel{\psi_*}{\longrightarrow}
{\rm Br}(M^{\delta}_C(G)^{\rm rs}) \longrightarrow 0\, ,
$$
where $p$ is the morphism in \eqref{p11},
${\rm wt}$ is the weight defined earlier, and $\psi_*$
takes a homomorphism $\eta\,:\,Z_G\,
\longrightarrow\, {\mathbb G}_m$ to the image of $\psi$ under the
homomorphism $H^2(M,\,Z)\longrightarrow H^2(M,\, {\mathbb G}_m)$
induced by $\eta$.
\end{theorem}

\section{Brauer group of moduli: $G$ is not simply
connected}\label{sec6}

In this section we compute the Brauer group of the moduli stack and
also that of the smooth locus of the moduli space of principal 
$G$-bundles when the semisimple group $G$ is not necessarily 
simply connected.

Let ${\widetilde G}$ be the universal cover of the semisimple
group $G$. The fundamental group $\pi_1(G)$ is a subgroup of the
center
$$
Z_{\widetilde G}\, \subset\, \widetilde G\, .
$$
Fix a maximal torus $\widetilde{T}\,\subset\, \widetilde G$.
We fix an isomorphism 
$$
\rho\,:\, \widetilde{T} \,\longrightarrow\, ({\mathbb G}_m)^\sigma~\,
~\, \text{~with~}~\, \rho(\pi_1(G)) \,=\,
\prod_{i=1}^\sigma \mu_{r_i}\, .
$$

There is a canonical isomorphism $\pi_0({\mathcal M}_C(G))
\,=\,\pi_0(L(G))\,\cong\, 
\pi_1(G)$. Take any $$\delta\,\in \,\pi_1(G)\, .$$ Let
$$
L(G)^{\delta}\, \subset\, L(G)~\,~\text{~and~}
~\,~
{\mathcal M}_C(G)^{\delta}\, \subset\, {\mathcal M}_C(G)
$$
be the connected components corresponding to $\delta$.

The above notation ${\mathcal M}_C(G)^{\delta}$ for a
connected component should not be confused with the previous
notation ${\mathcal M}^{\delta}_C(G)$ for a twisted moduli
stack. It should also be clarified that the twisted moduli
stack (or the twisted moduli space) is defined only for
simply connected groups.

The following proposition is proved in \cite{BLS}; see
\cite[p. 186, Proposition 1.3]{BLS}.

\begin{proposition}\label{unifG}
Take any $\delta \,\in \,\pi_1(G)$. For any 
$\zeta \,\in\, L(G)^{\delta}({\mathbb C})$, there is a natural 
isomorphism
$$
{\mathcal M}_C(G)^{\delta}\,\cong\, (\zeta^{-1}L_C(G)\zeta)
\backslash {\mathcal Q}_{{\widetilde G}}\, .
$$
\end{proposition}

Let
\begin{equation}\label{qdg}
q^{\delta}_G\,:\, {\mathcal Q}_{{\widetilde G}}\,
\longrightarrow\, {\mathcal M}_C(G)^{\delta}
\end{equation}
be the quotient morphism in Proposition \ref{unifG}.

Let
\begin{equation}\label{qdg2}
\widehat{q}^{\delta}_{\widetilde G}\,:\, {\mathcal Q}_{{\widetilde G}}
\, \longrightarrow\, {\mathcal M}^{\delta}_C({\widetilde G})
\end{equation}
be the quotient morphism to the twisted moduli stack in
Proposition \ref{unifG0}. There is a natural morphism of stacks
\begin{equation}\label{mul}
\gamma \,:\,{\mathcal M}^{\delta}_C({\widetilde G})\,\longrightarrow\,
{\mathcal M}_C(G)^{\delta}\, .
\end{equation}
This morphism $\gamma$ takes the semistable locus in
${\mathcal M}^{\delta}_C({\widetilde G})$ to the
semistable locus in ${\mathcal M}_C(G)^{\delta}$; this is an
immediately consequence of \cite[p. 319, Proposition 3.17]{Ra}.
Hence $\gamma$ induces a morphism
\begin{equation}\label{mul2}
\gamma_1\,:\, M^\delta_C({\widetilde G})\,\longrightarrow\,
M_C(G)^{\delta}
\end{equation}
between the corresponding coarse moduli spaces of semistable
bundles.

Let
$$
M_C(G)^{\delta, {\rm rs}}\, \subset\, M_C(G)^{\delta}
$$
be the smooth locus; we recall that $M_C(G)^{\delta, {\rm rs}}$
parametrizes the regularly stable bundles in $M_C(G)^{\delta}$
(see \cite[Corollary 3.6]{BH2}). Define the finite group
\begin{equation}\label{gam}
\Gamma\, :=\, H^1(C,\, \pi_1(G))\, .
\end{equation}

\begin{lemma}\label{lema}
\begin{enumerate}
\item The morphism $q_G^{\delta}$ in \eqref{qdg} factors as follows:
$$ 
{\mathcal Q}_{{\widetilde G}} \,\stackrel{\widehat{q}_{\widetilde 
G}^{\delta}} 
{\longrightarrow} \, {\mathcal M}^{\delta}_C({\widetilde G})
\,\stackrel{\gamma}{\longrightarrow}\,
{\mathcal M}_C(G)^{\delta}\, ,
$$
where $\widehat{q}_{\widetilde G}^{\delta}$ and $\gamma$ are
defined in \eqref{qdg2} and \eqref{mul} respectively.

\item The restriction of the morphism $\gamma_1$ in
\eqref{mul2} to the Zariski open subset
$$\gamma_1^{-1}(M_C(G)^{\delta, {\rm rs}}) \,\subset\,
M^{\delta}_C({\widetilde G})^{\rm rs}$$ defines a principal
$\Gamma$--bundle over $M_C(G)^{\delta, {\rm rs}}$, where
$\Gamma$ is defined in \eqref{gam}. 
\end{enumerate}
\end{lemma}

\begin{proof}
See \cite[p. 189, (2.4)]{BLS} for the first part. The second
part follows from the definition of regularly stable bundles.
\end{proof}

\begin{theorem}\label{th-m-s}
The Brauer group computation for $M_C(G)^{\delta, {\rm rs}}$
and $M^{\delta}_C({\widetilde G})^{\rm rs}$ are related by the
following exact sequence:
$$
0\, \longrightarrow\, H^1(\Gamma, \,{\mathbb C}^*)\,\longrightarrow\,
H^1(M_C(G)^{\delta, {\rm rs}},\, {\mathcal O}^*) \,\longrightarrow\, 
H^1(M^{\delta}_C({\widetilde G})^{\rm rs},\,{\mathcal O}^*)
$$
$$
\longrightarrow \,
H^2(\Gamma, \,{\mathbb C}^*)\, \longrightarrow \, 
{\rm Br}(M_C(G)^{\delta, {\rm rs}})\,\longrightarrow \, 
{\rm Br}(M^{\delta}_C({\widetilde G})^{\rm rs})\,\longrightarrow
\, 0\, .
$$
\end{theorem}

\begin{proof}
Consider the open subset $\gamma_1^{-1}(M_C(G)^{\delta, {\rm 
rs}}) \,\subset\, M^{\delta}_C({\widetilde G})^{\rm rs}$ in
Lemma \ref{lema}. We will show that its complement is
of codimension at least two. For that first note that
the morphism $\gamma_1$ in \eqref{mul2} is finite because
for any ample line bundle $L_0$ on $M^{\delta}_C({\widetilde 
G})^{\rm rs}$, the pullback $\gamma^*_1 L_0$ is ample. Therefore,
the fact that the codimension of the complement of the open
subset
$M_C(G)^{\delta, {\rm rs}}\, \subset\, M_C(G)$ is at least two
implies that the codimension of the complement of
$\gamma_1^{-1}(M_C(G)^{\delta, {\rm rs}})$ is at least two.

We can apply the Serre spectral sequence to the
principal $\Gamma$--bundle
$$
\gamma_1\, :\, \gamma_1^{-1}(M_C(G)^{\delta, {\rm rs}}) 
\,\subset\, M^{\delta}_C({\widetilde G})^{\rm rs}
$$
in Lemma \ref{lema} and get the following exact sequence:
$$
0\, \longrightarrow\,
H^1(\Gamma, \,{\mathbb C}^*)\, \longrightarrow \,
H^1(M_C(G)^{\delta, {\rm rs}},\, {\mathcal O}^*) \,\longrightarrow\, 
H^1(M^{\delta}_C({\widetilde G})^{\rm rs},\,{\mathcal O}^*)
$$
$$
\longrightarrow \,
H^2(\Gamma, \,{\mathbb C}^*)\, \longrightarrow \, 
{\rm Br}(M_C(G)^{\delta, {\rm rs}})\,\longrightarrow \, 
{\rm Br}(M^{\delta}_C({\widetilde G})^{\rm rs}) \, .
$$
It remains to show that the above homomorphism
\begin{equation}\label{t1}
{\rm Br}(M_C(G)^{\delta, {\rm rs}})\, \longrightarrow \,
{\rm Br}(M^{\delta}_C({\widetilde G})^{\rm rs})
\end{equation}
is surjective.

Let ${\mathcal M}^{\delta}_C({\widetilde G})'
\, \subset\,{\mathcal M}^{\delta}_C({\widetilde G})^{\rm rs}$
be the open sub-stack that lies over 
$\gamma_1^{-1}(M_C(G)^{\delta, {\rm rs}})$.
The moduli stack ${\mathcal M}_C(\pi_1(G))$ of principal
$\pi_1(G)$--bundles on $C$ acts naturally on
${\mathcal M}^{\delta}_C({\widetilde G})'$. We note
that ${\mathcal M}_C(\pi_1(G))$ is a gerbe over $\Gamma$ (defined
in \eqref{gam}) banded by $\pi_1(G)$. Since $\Gamma$
is a finite group, this gerbe is isomorphic to
$\Gamma\times B(\pi_1(G))$. Fixing such an isomorphism,
we construct an action of $\Gamma$ on
${\mathcal M}^{\delta}_C({\widetilde G})'$
by restricting the action of ${\mathcal M}_C(\pi_1(G))$.
The quotient for this action
$$
{\mathcal M}^{\delta}_C({\widetilde G})'/\Gamma\,
\longrightarrow\,
\gamma_1^{-1}(M_C(G)^{\delta, {\rm rs}})/\Gamma\,=\,
M_C(G)^{\delta, {\rm rs}}
$$
is a gerbe banded by $\pi_1(G)$. 

Recall that the homomorphism $\psi_*$ in Theorem
\ref{exact-gerbe} is surjective, so ${\rm 
Br}({\mathcal M}^{\delta}_C({\widetilde 
G})^{\rm rs})$ is constructed from the gerbe
$$
{\mathcal M}^{\delta}_C({\widetilde G})^{\rm rs}
\, \longrightarrow\, {M}^{\delta}_C({\widetilde G})^{\rm rs}
$$
(which is banded by $Z_{\widetilde G}$)
using ${\rm Hom}(Z_{\widetilde G},{\mathbb G}_m)$.
We have shown that this gerbe descends to
$M_C(G)^{\delta, {\rm rs}}$. Hence the homomorphism
in \eqref{t1} is surjective.
\end{proof}

Let ${\mathcal M}_C(G)^{\delta, {\rm rs}}$ be the open sub--stack
of ${\mathcal M}_C(G)^{\delta}$ defined by the regularly stable
principal bundles. The morphism to the coarse moduli space
\begin{equation}\label{p}
p\,:\, {\mathcal M}_C(G)^{\delta, {\rm rs}}\,\longrightarrow\, 
M_C(G)^{\delta,{\rm rs}}
\end{equation}
is a gerbe banded by $Z_G$, and we have the induced homomorphism 
\begin{equation}\label{p2}
p^*\,:\,{\rm Br}(M_C(G)^{\delta, {\rm rs}})\,\longrightarrow
\, {\rm Br}({\mathcal M}_C(G)^{\delta})\, .
\end{equation}
Now the exact sequence in Lemma 
\ref{pic-bra} can be used in relating the Brauer group of the smooth
locus of the moduli space to the Brauer group of the
stack. The kernel of the homomorphism $p^*$ in \eqref{p2}
can be computed using a result of \cite{BH} which will be
recalled below.

Let $\Psi \,\subset\, {\rm Hom}(Z_{\widetilde G}\otimes_{\mathbb Z} 
Z_{\widetilde G}\, , {\mathbb Q}/{\mathbb Z})$ be the abelian
group of all symmetric bilinear maps
$$
b\,:\, Z_{\widetilde G}\times Z_{\widetilde G} 
\,\longrightarrow\,{\mathbb Q}/{\mathbb Z}
$$
that come from an even $W$--invariant symmetric bilinear form
$$
\Lambda_{\rm coroots}\times \Lambda_{\rm coroots}\,\longrightarrow
\,{\mathbb Z}\, .
$$ 
Let
\begin{equation}\label{PsiG}
\Psi(G)\,\subset\, \Psi
\end{equation}
be the subgroup consisting of all
elements $b$ such that $b(\pi_1(G)\times \pi_1(G))\,=\,0$.

Given an element $\delta \,\in \,\pi_1(G)$, let 
\begin{equation}\label{evG}
{\rm ev}_G^{\delta}\,:\, \Psi(G) \,\longrightarrow\, {\rm
Hom}(Z_{\widetilde G}/{\pi_1(G)}\, ,\, {\mathbb Q}/{\mathbb Z})
\end{equation}
be the evaluation map that sends any $b$ to
$b(\delta,-)\,:\,Z_G\,\longrightarrow \,{\mathbb 
Q}/{\mathbb Z}$.

The following proposition (Proposition 7.4 in \cite{BH}) computes 
the weight map.

\begin{proposition} \label{weight}
The kernel of the homomorphism 
$$p^*\,:\,{\rm Br}(M_C(G)^{\delta,{\rm rs}})\,\longrightarrow\, 
{\rm Br}({\mathcal M}_C(G)^{\delta})$$ is given by the 
${\rm Coker}({\rm ev}_C^{\delta})$.
\end{proposition}

This proposition
and Theorem \ref{gerbe-sc} together imply the following:

\begin{corollary}\label{sc}
If $G$ is almost simple and simply connected, then
${\rm Br}(M_C(G)^{\rm rs})\,=\,Z^{\vee}_G$.
\end{corollary}

\section{A direct approach}

In this section we pursue the earlier spectral sequence argument.

Let ${\widetilde G}\longrightarrow G$ be the universal cover,
where $G$ is semisimple. The
kernel of this homomorphism is identified with
$\pi_1(G)$. Define the finite group
$$
\Gamma \,:=\, H^1(C,\, \pi_1(G))\, .
$$

\begin{lemma} \label{LCGnsc}
There is a short exact sequence 
$$ 
0\longrightarrow L_C({\widetilde G})/ \pi_1(G) \longrightarrow
L_C(G)\longrightarrow \Gamma \longrightarrow 0 \, .
$$

The quotient space $L_C({\widetilde G})/\pi_1(G)$ has the
homotopy type of $G\times \Omega({\widetilde G})^{2g-1}$
\end{lemma}

\begin{proof}
Consider the short exact sequence of groups
$$
e\, \longrightarrow\, \pi_1(G) \, \longrightarrow\,
{\widetilde G}\, \longrightarrow\, G \, \longrightarrow\, e\, .
$$
We have $H^1(C-p_0,\, G)\,=\, 0$ \cite{Ha}, also $H^1(C-p_0,
\, \pi_1(G))\,=\, H^1(C,\, \pi_1(G))$. Therefore, from the
long exact sequence of cohomologies associated to it, the
exact sequence in the lemma is obtained.

The statement on homotopy type of $L_C({\widetilde G})/\pi_1(G)$
is a consequence of the fact that
$L_C({\widetilde G})$ itself has the homotopy type of ${\widetilde
G}\times\Omega({\widetilde G})^{2g-1}$ (see Theorem \ref{thm-T}).
\end{proof}

\begin{lemma}\label{L-pi}
We have
\begin{enumerate}
\item $H^1(B(L_C({\widetilde G})/ \pi_1(G)),\,{\mathbb C}^*)\,=\,0$, and

\item $H^2(B(L_C({\widetilde G})/ \pi_1(G)),\, {\mathbb C}^*)
\,=\,\pi_1(G)^{\vee}$.
\end{enumerate}
\end{lemma}
\begin{proof}
Since $L_C({\widetilde G})$ is connected (see Lemma \ref{lem00}),
$$
\pi_1(B(L_C({\widetilde G})/ \pi_1(G)))\,=\,
\pi_0(L_C({\widetilde G})/ \pi_1(G))\,=0\, .
$$
Hence the first statement follows.

To prove the second statement, consider the product decomposition
$$
B(L_C({\widetilde G})/ \pi_1(G)) \,=\,BG\times B(\Omega {\widetilde 
G})^{2g-1}\, ,
$$
and apply the K\"unneth decomposition to it.
The individual cohomology computations are done as follows.
We have $H^1(BG,\,{\mathbb C}^*)\,=\,0$ and 
$H^1(B(\Omega {\widetilde G}),\,{\mathbb C}^*)\,=\,0$
because $\pi_1(BG)\,=\,\pi_0(G)\,=\,0$; also, 
$$
\pi_1(B(\Omega {\widetilde G}))\,=\,\pi_0(\Omega {\widetilde G})
\,=\,\pi_1({\widetilde G})\,=\,0 \, .
$$

These and Hurewicz isomorphism together imply that 
$$
H^2(B(\Omega {\widetilde G}),\,{\mathbb C}^*)\,\cong\,
{\rm Hom}(\pi_2(B(\Omega {\widetilde G})),\,{\mathbb C}^*)
~\, \text{~and~}\,~ 
H^2(BG,\,{\mathbb C}^*)\,\cong\,
{\rm Hom}(\pi_2(BG),\,{\mathbb C}^*)\, .
$$
Since $\pi_2(B(\Omega {\widetilde G}))\,=\,
\pi_1(\Omega {\widetilde G})\,=\,\pi_2({\widetilde G})\,=\,0$, and 
$\pi_2(BG)\,=\,\pi_1(G)$, the second statement in the
proposition follows.
\end{proof}

The following is a generalization of Proposition \ref{LCG}.

\begin{proposition}\label{LCG-2}
With the above notation and $\Gamma \,:=\, H^1(C,\, \pi_1(G))$,
\begin{enumerate}
\item $H^1(BL_C(G),\,{\mathbb Z})\,=\,0$, and
\item there is a short exact sequence
\begin{equation}\label{BLC} 
0 \, \longrightarrow \, H^2(\Gamma,\, {\mathbb C}^*)\,\longrightarrow\,
H^2(BL_C(G),\,{\mathbb C}^*))\,\longrightarrow\,
\pi_1(G)^{\vee}\,\longrightarrow\,0\, . 
\end{equation} 
\end{enumerate}
\end{proposition}

\begin{proof}
The first part is a consequence of the fact that 
$$
\pi_1(BL_C(G))\,=\,\pi_0(L_C(G))\,=\,\Gamma
$$ 
is a finite group.

For the second part, we use Lemma \ref{LCGnsc} to realize the 
space 
$B(L_C({\widetilde G})/\pi_1(G))$ as a principal $\Gamma$--bundle
over $BL_C(G)$. More precisely,
$$
B(L_C({\widetilde G})/\pi_1(G))
\,=\, (EL_C(G))/(L_C({\widetilde G})/\pi_1(G))\,\longrightarrow
\, EL_C(G)/L_C(G)\,=\, BL_C(G)\, .
$$
The Serre spectral sequence gives the following exact sequence 
$$
H^0(\Gamma,\, H^1(B(L_C({\widetilde G})/\pi_1(G)),{\mathbb C}^*) ) 
\,\longrightarrow\,H^2(\Gamma, \,{\mathbb C}^*)\,\longrightarrow
\,{\rm kernel}[H^2(BL_C(G), {\mathbb C}^*)\rightarrow
$$
$$
H^0(\Gamma,\, H^2(B(L_C({\widetilde G})/\pi_1(G)),{\mathbb C}^*))]
\,\longrightarrow\, H^1(\Gamma,\,
H^1(B(L_C({\widetilde G})/\pi_1(G)),{\mathbb C}^*))\, .
$$
As $H^1(B(L_C({\widetilde G})/\pi_1(G)),{\mathbb C}^*)\,=\,0$
(see Lemma \ref{L-pi}(1)), this exact sequence reduces to
an isomorphism
\begin{equation}\label{sr}
H^2(\Gamma,\, {\mathbb C}^*)\,\stackrel{\sim}{\longrightarrow}\,
{\rm kernel}[H^2(BL_C(G),\, {\mathbb C}^*)\rightarrow
H^0(\Gamma, H^2(B(L_C({\widetilde G})/\pi_1(G)),\,{\mathbb C}^*))]\, .
\end{equation}
We have $H^0(\Gamma,\, H^2(B(L_C({\widetilde G})/\pi_1(G)),{\mathbb 
C}^*))\,=\, H^2(B(L_C({\widetilde G})/\pi_1(G)),{\mathbb C}^*)
\,=\, \pi_1(G)^{\vee}$ by Lemma \ref{L-pi}(2). Therefore,
to complete the proof it suffices to show that the homomorphism
\begin{equation}\label{y1}
H^2(BL_C(G), {\mathbb C}^*)\,\longrightarrow\,
H^0(\Gamma,\, H^2(B(L_C({\widetilde G})/\pi_1(G)),{\mathbb C}^*))
\,=\, \pi_1(G)^{\vee}
\end{equation}
in \eqref{sr} is surjective.

Fix a point $x_0\, \in \, C\setminus\{p_0\}$, where $p_0$
is the base point in \eqref{llc}. Construct the group scheme
$$
{\mathcal H}\, :=\, L_C(G)\times_G {\widetilde G}
$$
using the evaluation $L_C(G)\, \longrightarrow\, G$ at $x_0$. Both
the natural morphisms
$$
B(L_C({\widetilde G}))\, \stackrel{a}{\longrightarrow}\, 
B(L_C({\widetilde 
G})/\pi_1(G)) \, ~\, \text{~and~}\, ~\,
B({\mathcal H})\, \stackrel{b}{\longrightarrow}\,B(L_C(G))
$$
are gerbes banded by $\pi_1(G)$, while the natural morphism
$B(L_C({\widetilde G})/\pi_1(G))\, 
\stackrel{c}{\longrightarrow}\,B(L_C(G))$ is an \'etale Galois
covering with Galois group $\Gamma$. Note that we have a
homomorphism
$$
L_C({\widetilde G}) \, \longrightarrow\,
{\mathcal H}\,=\, L_C(G)\times_G {\widetilde G}
$$
constructed using the evaluation $L_C({\widetilde G})\,\longrightarrow
\, {\widetilde G}$ at $x_0$. This homomorphism produces a morphism
$$
d\, :\, B(L_C({\widetilde G}))\, \longrightarrow\,B({\mathcal H})\, ,
$$
and the diagram
$$
\begin{matrix}
B(L_C({\widetilde G})) &\stackrel{d}{\longrightarrow} &
B({\mathcal H})\\
~\Big\downarrow  a && ~\Big\downarrow b \\
B(L_C({\widetilde G})/\pi_1(G)) &\stackrel{c}{\longrightarrow} &
B(L_C(G))
\end{matrix}
$$
is clearly commutative. It is straight-forward to check that $d$ 
produces an isomorphism of the gerbe $B(L_C({\widetilde G}))$ with
the pulled back gerbe $c^*B({\mathcal H})$.

The gerbe $B(L_C({\widetilde G}))\, \stackrel{a}{\longrightarrow}
\,B(L_C({\widetilde G})/\pi_1(G))$ and a character of $\pi_1(G)$
together produce a gerbe on $B(L_C({\widetilde G})/\pi_1(G))$ 
banded by ${\mathbb C}^*$. In view of Lemma \ref{L-pi}(2), we
know that all elements of $H^2(B(L_C({\widetilde 
G})/\pi_1(G)),{\mathbb C}^*))$ arise this way. Since the gerbe
$B(L_C({\widetilde G}))\, \stackrel{a}{\longrightarrow}\,
B(L_C({\widetilde G})/\pi_1(G))$ is the pullback of a gerbe on
$B(L_C(G))$, it follows that the homomorphism in \eqref{y1}
is surjective.
\end{proof}

The analogue of Theorem \ref{thm0} for a general semisimple
group is the following.

\begin{theorem}\label{direct}
For any $\delta \in \pi_1(G)$, the following exact 
sequence computes the Brauer group of the moduli stack of
principal $G$-bundles over $C$:
$$
{\rm Pic}({\mathcal M}_C(G)^{\delta}) \, \longrightarrow\,
{\rm Pic}({\mathcal Q}_{\widetilde G})\, \longrightarrow\,
H^2(BL_C(G),\, {\mathbb C}^*)\, \longrightarrow\,
{\rm Br}({\mathcal M}_C(G)^{\delta})\,\longrightarrow\, 0\, ;
$$
the above group $H^2(BL_C(G),\, {\mathbb C}^*)$ is computed by
the exact sequence in (\ref{BLC}).
\end{theorem}

\begin{proof}
The proof of the theorem follows the same steps as in the
proof of Theorem \ref{thm0}, namely the descent spectral sequence
and the cohomology computations --- which are now provided
by Proposition \ref{LCG-2}.
\end{proof}

As the last step in the computation we prove the following.

\begin{proposition}\label{gerbe-final}
The homomorphism
$$
p^*\,:\, {\rm Br}(M_C(G)^{\delta,{\rm rs}}) 
\,\longrightarrow \,
{\rm Br}({\mathcal M}_C(G)^{\delta,{\rm rs}})\,=\,
{\rm Br}({\mathcal M}_C(G)^\delta)
$$
in \eqref{p2} is surjective.
\end{proposition}

\begin{proof}
The above equality ${\rm Br}({\mathcal M}_C(G)^{\delta,{\rm rs}})
\,=\, {\rm Br}({\mathcal M}_C(G)^\delta)$ follows
immediately from Proposition \ref{bis-hoff}.

Let ${\widetilde G}\,\longrightarrow\, G$ be
the universal covering homomorphism; its kernel is
$\pi_1(G)$. We get 
an inclusion $\pi_1(G)\,\subset\,
Z_{{\widetilde G}}$, and the quotient is identified
with $Z_G$. There is a canonical inclusion of $Z_G$
$$
Z_G\,\hookrightarrow\, L_C({\widetilde G})/\pi_1(G)
\,\hookrightarrow\, L_C(G)
$$
(see Lemma \ref{LCGnsc}).

Let
$$
{\mathcal Q}_1\subset {\mathcal Q}_{\widetilde G}
$$
be the open sub--stack
defined by the inverse image of ${\mathcal M}_C(G)^{\delta,{\rm rs}}$ 
under the morphism $q^\delta_G$ in \eqref{qdg}. Choosing an element
$\zeta \in (LG)^{\delta}({\mathbb C})$
as in Proposition \ref{unifG},
we get an action of the group $L_C(G)/Z_G$ on ${\mathcal Q}_1$ 
such that the corresponding quotient is $M_C(G)^{\delta,{\rm 
rs}}$. 

Consider the morphism $p$ in \eqref{p}. Since the fiber product
${\mathcal Q}_1\times_{M_C(G)^{\delta,{\rm rs}}}
{\mathcal M}_C(G)^{\delta, {\rm rs}}$
is identified with ${\mathcal Q}_1 \times Z(G)$,
we conclude, by taking cohomologies, that the diagram
$$
\begin{matrix}
 H^2(BL_C(G),\,{\mathbb C}^*)&\longrightarrow & 
{\rm Br}({\mathcal M}_C(G)^{\delta,{\rm rs}}) \\
 \Big\uparrow & & \Big\uparrow \\
H^2(B(L_C(G)/Z(G)),\,{\mathbb C}^*)&\longrightarrow & 
{\rm Br}(M_C(G)^{\delta,{\rm rs}})
\end{matrix}
$$
is commutative. The upper horizontal arrow is surjective by Theorem
\ref{direct}. In view of the above diagram, to complete the
proof of the proposition it suffices to show that the homomorphism
$$
H^2(B(L_C(G)/Z(G)),\,{\mathbb C}^*)\,\longrightarrow\,
H^2(BL_C(G),\,{\mathbb C}^*)
$$ 
is surjective. But this surjectivity follows by applying Proposition
\ref{LCG-2}(2) to $BL_C(G)$ and $B(L_C(G)/Z(G))$ and observing 
that the terms in the resulting two exact
sequences match. This completes the proof of the proposition.
\end{proof}

Combining the above proposition with Proposition \ref{weight} and
Lemma \ref{pic-bra} we get the following.

\begin{corollary}\label{corl}
For any semisimple $G$, there is a short exact sequence
$$
0\,\longrightarrow \,{\rm Coker}({\rm 
ev}_G^{\delta})\,\longrightarrow\,
{\rm Br}(M_C(G)^{\delta,{\rm rs}}) \,\longrightarrow \,{\rm 
Br}({\mathcal M}_C(G)^\delta)\,\longrightarrow \,0\, .
$$
\end{corollary}

\section{Computations for classical groups}\label{sec7}

\subsection{The cases of ${\rm SL}_n$ and ${\rm PGL}_n$}

Take a positive integer $n$, and take any $d\, \in\, [0\, ,n-1]$.
Fix a line bundle $L_d$ on $C$ of degree $d$. The twisted moduli
stack ${\mathcal M}^d_C({{\rm SL}_n})$ is the moduli stack of
vector bundles on $C$ of rank $n$ and determinant isomorphic to
$L_d$. Theorem \ref{thm0} and Theorem
\ref{thm0-d} imply that $${\rm Br}({\mathcal M}^d_C({{\rm SL}_n}))
\,=\,0$$ for any $d$. 

By Proposition \ref{weight}, the Brauer group of
$M^{d}_C({{\rm SL}_n})^{\rm rs}$
coincides with the cokernel of the homomorphism ${\rm ev}_G^{\delta}$,
and this
can be computed to be ${\mathbb Z}/{\rm g.c.d.}(n,d){\mathbb Z}$
by Table 1 of \cite{BH}. This recovers the results of \cite{BBGN}.

For the case of ${\rm PGL}_n$, the exact sequence given in Theorem
\ref{th-m-s} or Theorem \ref{direct} reproduces Theorem 1 of
\cite{BHO}.

The above approach also works for quotients of ${\rm SL}_n$ by
finite subgroups of the center $Z_{{\rm SL}_n}$. 

\subsection{The cases ${\rm Sp}_{2n}$ and ${\rm PSp}_{2n}$} 
We have 
${\rm Br}({\mathcal M}^d_C({{\rm Sp}_{2n}}))\,= \,0$ for
$d\,=\,0\, ,1$ by Theorem \ref{thm0} and Theorem \ref{thm0-d}. 

{}From Corollary \ref{sc} we have
${\rm Br}(M^{0}_C({\rm Sp}_{2n})^{\rm rs})\,=\,{\mathbb Z}/2{\mathbb 
Z}$.

\begin{proposition}\label{SP} 
Let $d=1$.
\begin{enumerate}
\item Assume that $n\, \ge \, 3$ is odd. Then 
$$
{\rm Br}(M^{1}_C({\rm Sp}_{2n})^{\rm rs})\,=\,0\, .
$$
The
square of the generating line bundle
on ${\mathcal M}^{1}_C({\rm Sp}_{2n})$ descends to $M^{1}_C({\rm 
Sp}_{2n})^{\rm rs}$ and generates ${\rm Pic}(M^{1}_C({\rm 
Sp}_{2n}))^{\rm rs}$.

\item Assume that $n\, \ge \, 3$ is even. Then 
${\rm Br}(M_C^{1}({\rm Sp}_{2n})^{\rm rs})\,=\,{\mathbb Z}/2{\mathbb 
Z}$.
The generating line bundle
on ${\mathcal M}^{1}_C({\rm Sp}_{2n})$ descends to $M^{1}_C({\rm
Sp}_{2n})^{\rm rs}$ and generates ${\rm Pic}(M^{1}_C({\rm 
Sp}_{2n})^{\rm rs})$.
\end{enumerate}
\end{proposition}

\begin{proof} This follows by the combination of exact sequences
in Theorem \ref{exact-gerbe} and Proposition \ref{weight}. 
The cokernel of the homomorphism ${\rm ev}_G^{\delta}$ is $0$ 
when $n$ is odd, and it is ${\mathbb Z}/2{\mathbb Z}$ when $n$ is 
even by Table 1 of \cite{BH}.
\end{proof}

\begin{corollary}\label{cor1}
Let $M^{1}_C({\rm Sp}_{2n})^{\rm ss}$ be the twisted moduli space
of semistable ${\rm Sp}_{2n}$--bundles. The variety $M^1_C({\rm 
Sp}_{2n})^{\rm ss}$ is locally factorial for odd $n$, and
it is not locally factorial for even $n$.
\end{corollary}

\begin{proof}
The Picard group of  $M^1_C({\rm Sp}_{2n})^{\rm ss}$ is always 
generated by the descent of the square of the generating line 
bundle on ${\mathcal M}^1_C({\rm Sp}_{2n})$ (see \cite[p. 209, 
Proposition 11.2(a)]{BLS}). Therefore, the corollary
follows from Proposition \ref{SP}.
\end{proof}

Note that $\pi_1({\rm PSp}_{2n})\, =\, Z_{{Sp}_{2n}}\,=\,
{\mathbb Z}/2{\mathbb Z}$.
For $d\, \in\, \pi_1({\rm PSp}_{2n})$,
let $${\mathcal M}_C({\rm PSp}_{2n})^d\, \subset\,
{\mathcal M}_C({\rm PSp}_{2n})$$ be the connected component
of the moduli stack of principal
${\rm PSp}_{2n}$--bundles corresponding to $d$. Similarly,
$$
M_C({\rm PSp}_{2n})^{d,{\rm rs}}\, \subset\,
M_C({\rm PSp}_{2n})^{\rm rs}
$$
is the connected component of the moduli space corresponding to 
the element $d$. The locus ${\mathcal M}_C({\rm PSp}_{2n}
)^{d,{\rm rs}}$ in ${\mathcal M}_C({\rm PSp}_{2n})^d$ of 
regularly stable bundles is an open sub--stack such that the
complement is of codimension at least two. Hence the Brauer
groups of ${\mathcal M}_C({\rm PSp}_{2n})^d$ and ${\mathcal M}_C
({\rm PSp}_{2n})^{d,{\rm rs}}$ coincide.

Define
$$
\Gamma\, :=\, H^2(C,\, {\mu}_2)
$$
(the 2--torsion points of the Jacobian of $C$).

\begin{proposition} Let $d=0$ or $1$. If $n\,\ge \,3$ is odd, then
${\rm Br}({\mathcal M}_C({\rm PSp}_{2n})^d)\,\cong\,
H^2(\Gamma, \,{\mathbb C}^*)$. If $n \ge 3$ is even, then
${\rm Br}({\mathcal M}_C({\rm PSp}_{2n})^d)\,\cong\,
H^2(\Gamma,\, {\mathbb C}^*)\oplus ({\mathbb Z}/2{\mathbb Z})$.
\end{proposition}

\begin{proof}
Let $U^{d,{\rm rs}}\,\subset\, M^{d}_C({\rm Sp}_{2n})^{\rm rs}$
be the Zariski open subset defined by the inverse image of
$M_C({\rm PSp}_{2n})^{d,{\rm rs}}$. The codimension of the complement
of $U^{d, {\rm rs}}$ is at least two. To prove the proposition
one can use both the approaches used here. For example, 
Theorem \ref{th-m-s} gives the following exact sequence
\begin{equation}\label{mm}
0 \,\longrightarrow\, {\mathbb Z}/m{\mathbb Z}\,\longrightarrow\,
H^2(\Gamma, \,{\mathbb C}^*)\,\longrightarrow\,
{\rm Br}(M_C({\rm PSp}_{2n})^{d,{\rm rs}})\,\longrightarrow\,
{\rm Br}(U^{d, {\rm rs}})\, \longrightarrow\, 0\, ,
\end{equation}
where $m$ is the smallest power of the generating line bundle on 
$U^{d, {\rm rs}}$ which descends to the moduli space 
$M_C({\rm PSp}_{2n})^{d,{\rm rs}}$.

Since $M^{d}_C({\rm Sp}_{2n})^{\rm rs}$ is smooth, and the codimension
of the complement of $U^{d, {\rm rs}}$ is at least two, the
homomorphisms
$$
\text{Pic}(M^{d}_C({\rm Sp}_{2n})^{\rm rs})
\,\longrightarrow\, \text{Pic}(U^{d, {\rm rs}})
~\,~\text{~and~}~\,~ \text{Br}(M^{d}_C({\rm Sp}_{2n})^{\rm rs})
\,\longrightarrow\, \text{Br}(U^{d, {\rm rs}})
$$
induced by the inclusion $U^{d, {\rm rs}}\,\hookrightarrow\,
M^{d}_C({\rm Sp}_{2n})^{\rm rs}$ are isomorphisms.

When $n$ is even, by \cite[p. 191, Proposition 4.2]{BLS}, the 
generating line bundle on the affine Grassmannian
${\mathcal Q}_{{\rm Sp}_{2n}}$ descends all the way to 
$M_C({\rm PSp}_{2n})^{d,{\rm rs}}$. Hence in this case we have
$m\,=\,1$ in \eqref{mm}.

When $d\,=\,0$ with $n$ odd, the generating line bundle on 
the affine 
Grassmannian descends to $M^{0}_C({\rm Sp}_{2n})^{\rm rs}$
(see Proposition \ref{SP}), but it does not descend
to $M_C({\rm PSp}_{2n})^{0,{\rm rs}}$ by \cite[Proposition 
4.2]{BLS}. Hence in this case we have $m\,=\,2$ in \eqref{mm}.

If $d\,=\,1$ with $n$ odd, then the smallest power of the 
generating
line bundle on the affine Grassmannian which descends to 
$M^{1}_C({\rm Sp}_{2n})^{\rm rs}$ is $2$, and also $2$
is the smallest power of the generating
line bundle on the affine Grassmannian that descends to
$M_C({\rm PSp}_{2n})^{0,{\rm rs}}$ (again by
Proposition \ref{SP} and \cite[Proposition 4.2]{BLS}).
Hence $m\,= \,1$ in this case.

If $d\,=\,1$ with $n$ odd, then the above observation and Proposition 
\ref{SP} together imply the proposition. 

In the remaining cases we get a short exact sequence of the form
\begin{equation}\label{ex.}
0\,\longrightarrow\, ({\mathbb Z}/2{\mathbb Z})^b\,\longrightarrow\,
{\rm Br}(M_C({\rm PSp}_{2n})^{d,{\rm rs}})\,\longrightarrow\, {\mathbb 
Z}/2{\mathbb Z} \,\longrightarrow\, 0\, .
\end{equation}
To complete the proof of the proposition it suffices to show that the
sequence in \eqref{ex.} splits. The exact sequence 
splits if the generating Brauer class of 
$M^{d}_C({\rm Sp}_{2n})^{\rm rs}$ (when it is non--trivial) is
the pull back of a two--torsion class in ${\rm Br}(M_C({\rm 
PSp}_{2n})^{d,{\rm rs}})$. But this was proved in Section
\ref{sec6} (see the proof that the homomorphism in \eqref{t1}
is surjective).
\end{proof}

\subsection{The cases of ${\rm Spin}_n$, ${\rm SO}_n$
and ${\rm PSO}_n$}

The center $Z_{{\rm Spin}_n}$ of ${\rm Spin}_n$ is ${\mathbb 
Z}/2{\mathbb Z}$
when $n$ odd, and it is ${\mathbb Z}/4{\mathbb Z}$ or 
${\mathbb Z}/2{\mathbb Z}\times {\mathbb Z}/2{\mathbb Z}$ depending on
whether $n$ is of the form $4l+2$ or $4l$.
Take any $\delta\, \in\, Z_{{\rm Spin}_n}$.
Since the Dynkin index of the standard representation of ${\rm Spin}_n$ 
in ${\rm SL}_n$ is 2, the Picard group of the twisted moduli stack
${\mathcal M}^{\delta}_C({\rm Spin}_n)$ is generated by the 
Pffafian line bundle $P$
whose square is the determinant bundle of cohomology.

By Theorem \ref{thm0} and Theorem \ref{thm0-d},
$$
{\rm Br}({\mathcal M}^{\delta}_C({\rm Spin}_n))\,=\,0\, .
$$

\begin{proposition}\label{Spin}
\begin{enumerate}
\item If $\delta \,\in\, Z_{{\rm Spin_n}}$ is zero, then
${\rm Br}(M^{0}_C({\rm Spin}_{n})^{\rm rs})\,= \,Z^{\vee}_{{\rm 
Spin}_n}$.

\item Assume that $n\, \ge \, 4$ is odd. Then 
${\rm Br}(M^{1}_C({\rm Spin}_{n})^{\rm rs})\,=\,{\mathbb 
Z}/2{\mathbb Z}$. The Pffafian bundle descends to
$M^{1}_C({\rm Spin}_{n})^{\rm rs}$ and generates
the Picard group 
${\rm Pic}(M^{1}_C({\rm Spin}_{n})^{\rm rs})$.

\item Assume that $n\,=\,4l+2\,\ge \, 8$ and $\delta\,=\,1$ or $3$ {\rm 
mod} $(4)$. Then 
${\rm Br}(M^{\delta}_C({\rm Spin}_{n})^{\rm rs})\,=\,0$.
The fourth power of the Pffafian bundle descends to 
$M^{\delta}_C({\rm Spin}_{n})^{\rm rs}$ and it generates the
Picard group ${\rm Pic}(M^{\delta}_C({\rm Spin}_{n})^{\rm rs})$.

\item Assume that $n\,=\,4l+2\, \ge \,8$ and $\delta=2$ {\rm mod} $(4)$.
Then
${\rm Br}(M^{\delta}_C({\rm Spin}_{n})^{\rm rs})\,=\,
{\mathbb Z}/2{\mathbb Z}$. The square of the Pffafian bundle
descends to $M^{\delta}_C({\rm Spin}_{n})^{\rm rs}$ generating
${\rm Pic}(M^{\delta}_C({\rm Spin}_{n})^{\rm rs})$.

\item Assume that $n=4l \, \ge \, 8$ and $\delta \neq 0$. Then ${\rm Br}
(M^{\delta}_C({\rm Spin}_{n})^{\rm rs})\,=\,{\mathbb Z}/2{\mathbb 
Z}$. The square of the Pffafian bundle descends to $M^{\delta}_C(
{\rm Spin}_{n})^{\rm rs}$ generating the Picard group.
\end{enumerate}
\end{proposition}

\begin{proof}
In view of Proposition \ref{weight}, this is a straight--forward
calculation using Table 1 in \cite{BH}.
\end{proof}

For the case of ${\rm SO}_n$,
note that $\pi_1({\rm SO}_n)\, \subset\, Z_{{\rm Spin}_n}$.
For any $\delta\, \in\, \pi_1({\rm SO}_n)$, let
$$
{\mathcal M}_C({\rm SO}_{n})^{\delta}\, \subset\,
{\mathcal M}_C({\rm SO}_{n})
$$
be the connected component corresponding to $\delta$. Similarly,
let
$$
M_C({\rm SO}_{n})^{\delta, {\rm rs}}\, \subset\,
M_C({\rm SO}_{n})^{\rm rs}
$$
be the connected component corresponding to $\delta$.

Define
$$
\Gamma'\, :=\, H^1(C,\, \pi_1({\rm SO}_n)^{\vee})\, .
$$

\begin{proposition}\label{SOn}
Let $n\,\ge \, 8$. Take any $\delta \in \pi_1({\rm SO}_n)$.
\begin{enumerate}
\item If $\delta \,=\, 0$, then
$$
{\rm Br}(M_C({\rm SO}_{n})^{0,{\rm rs}})\,=\,
H^2(\Gamma', \,{\mathbb C}^*) \oplus Z^{\vee}_{{\rm Spin}_n}\, .
$$

\item If $\delta\,\not=\, 0$, then
${\rm Br}(M_C({\rm SO}_{n})^{\delta,{\rm rs}})\,\cong\, 
H^2(\Gamma',\, {\mathbb C}^*)\oplus {\mathbb Z}/2{\mathbb Z}$. 

\item We have ${\rm Br}({\mathcal M}_C({\rm SO}_{n})^{\delta})\,\cong
\,H^2(\Gamma',\, {\mathbb C}^*)\oplus {\mathbb Z}/2{\mathbb Z}$.
\end{enumerate}
\end{proposition}

\begin{proof}
The proof is very similar to that for $\text{Sp}_n$
as done in Proposition \ref{SP}.

For any $\delta\, \in\, \pi_1({\rm SO}_n)$, let
$$
U^{\delta,{\rm rs}}\,\subset\, M^\delta_C({\rm Spin}_{n})^{\rm rs}
$$
be the Zariski open subset given by the inverse image of
$M_C({\rm SO}_{n})^{\delta, {\rm rs}}$. The codimension of the 
complement
of $U^{\delta, {\rm rs}}$ in $M^\delta_C({\rm Spin}_{n})$
is at least two. Now Theorem \ref{th-m-s} gives the exact sequence
\begin{equation}\label{m2m}
0 \,\longrightarrow\, {\mathbb Z}/m{\mathbb Z}\,\longrightarrow\,
H^2(\Gamma', \,{\mathbb C}^*)\,\longrightarrow\,
{\rm Br}(M_C({\rm SO}_{n})^{\delta,{\rm rs}})\,\longrightarrow\,
{\rm Br}(U^{\delta, {\rm rs}})\, \longrightarrow\, 0\, ,
\end{equation}
where $m$ is the smallest power of the generating line bundle on 
$U^{\delta, {\rm rs}}$ which descends to 
$M_C({\rm SO}_{2})^{\delta,{\rm rs}}$.

We will show that the integer $m$ defined in \eqref{m2m} is $1$. This
amounts to showing that the generating line bundle on $M^\delta_C
({\rm Spin}_{n})^{\rm rs}$ descend to the quotient $M_C({\rm SO}_{
n})^{\delta, {\rm rs}}\,=\, U^{\delta,{\rm rs}}/\Gamma'$.

If $\delta \,=\,0$, this follows from the existence of the
Pffafian bundle.

Assume that $\delta \,\not=\,0$. We note that Proposition
\ref{Spin} gives the smallest power of the Pffafian
bundle which descends
to $M^{\delta}_C({\rm Spin}_{n})^{\rm rs}$ from the moduli stack
$\mathcal{M}^{\delta}_C({\rm Spin}_{n})$. Hence it is enough
to compute the smallest power of the Pffafian bundle on the 
moduli stack $\mathcal{M}^{\delta}_C({\rm Spin}_{n})$ that 
descends to $M_C({\rm SO}_{n})^{\delta, {\rm rs}}$. To calculate 
this, we again use Proposition \ref{weight} and the Table 1 of
\cite{BH}; we explicitly check that the smallest power of the
Pffafian bundle that descends to $M_C({\rm SO}_{n})^{\delta, {\rm 
rs}}$ coincide with the power that descends to $M^{\delta}_C({\rm 
Spin}_{n})^{\rm rs}$.

Let
\begin{equation}\label{m3m}
0 \,\longrightarrow\,
H^2(\Gamma', \,{\mathbb C}^*)\,\longrightarrow\,
{\rm Br}(M_C({\rm SO}_{n})^{\delta,{\rm rs}})\,\longrightarrow\,
Z_{{\rm Spin}_n}\, \longrightarrow\, 0\, ,
\end{equation}
be the exact sequence obtained from \eqref{m2m}.

To conclude parts (1) and (2) of the proposition, we still need to
show that the exact sequence in \eqref{m3m}
splits. The proof of this splitting
is identical to the case for $\text{Sp}_n$; the
details are omitted.

For the last part of the proposition, we use the existence
of Pffafian bundle to conclude that the homomorphism
$$
H^2(BL_C({\rm SO}_{n}),\, {\mathbb C}^*)\,\longrightarrow
\,{\rm Br}({\mathcal M}_C({\rm SO}_{n})^{\delta})
$$
in Theorem \ref{direct} is an isomorphism. Now Proposition
\ref{LCG-2} gives an exact sequence 
$$
0 \,\longrightarrow \,
H^2(\Gamma', \,{\mathbb C}^*)\,\longrightarrow \,
{\rm Br}({\mathcal M}_C({\rm SO}_{n})^{\delta}) \,\longrightarrow
\,{\mathbb Z}/2{\mathbb Z} \,\longrightarrow \, 0 \, .
$$
This sequence splits by using the argument for splitting
of the exact sequence in \eqref{m3m} and Proposition 
\ref{gerbe-final}.
\end{proof}

For the case of ${\rm PSO}_{2n}$, the components of the moduli 
space are parametrized by the center $Z_{{\rm Spin}_{2n}}$. 
Since ${\rm PSO}_{2n}$ is of adjoint type, the
Brauer groups of $M_C({\rm PSO}_{2n})^{\delta, {\rm rs}}$ and 
${\mathcal M}_C({\rm PSO}_{2n})^{\delta}$ can be identified as 
before. Define
$$
\Gamma_1\,:=\, H^1(C,\, Z^{\vee}_{{\rm Spin}_{2n}}) ~\,
\text{~and~}\, ~ B\,=\, H^2(\Gamma_1,\,{\mathbb C}^*)\, .
$$

Using the above methods and  the
description of the Picard and Brauer groups as done in
Proposition \ref{Spin} and \cite[Proposition
5.5]{BLS}, we get the following proposition.

\begin{proposition}
Let $n\, \ge \, 4$. We have 
${\rm Br}(M_C({\rm PSO}_{2n})^{\delta, {\rm rs}})\,=\, 
(B/A_{2n,\delta}) \oplus Z^{\vee}$,
where the groups $A_{2n,\delta}$ is computed as follows:
$A_{4n,\delta}\,=\,{\mathbb Z}/2{\mathbb Z}$ (respectively,
$A_{4n,\delta}\,=\,0$) if 
$\delta \,=\, 0$ (respectively, $\neq 0$), and
$A_{4n+2,\delta}\,=\,{\mathbb Z}/4{\mathbb Z}$ (respectively,
$A_{4n+2,\delta}\,=\,{\mathbb Z}/2{\mathbb Z}$) if $\delta$ is
$0$ mod $(4)$ (respectively, $2$ mod $(4)$), and
$A_{4n+2,\delta}\,=\, 0$ otherwise.
\end{proposition}

There is a remaining classical group $\Omega_{4n}$ for $n\, \ge \, 3$ 
defined by taking quotient of ${\rm Spin}_{4n}$ by a central 
subgroup of order $2$ different from the one that defines ${\rm 
SO}_{4n}$. There are 
two choices for such a sub-group and they define isomorphic groups. 

We can use the methods described above to determine the Picard and
Brauer groups of the moduli spaces of $\Omega_{4n}$ bundles, and 
we get the following.

\begin{proposition}
\begin{enumerate}
\item We have ${\rm Br}(M_C(\Omega_{4n})^{\delta,{\rm rs}})
\,\cong\, H^2(\Gamma,\, {\mathbb C}^*)\oplus {\mathbb 
Z}/2{\mathbb Z}$. The square of the Pffafian bundle descends to 
$M_C(\Omega_{4n})^{\delta,{\rm rs}}$ generating the Picard group.

\item We have ${\rm Br}(\mathcal{M}_C(\Omega_{4n})^\delta) \,\cong\,
H^2(\Gamma,\,{\mathbb C}^*)\oplus A_{4n,d}$, where 
$A_{4n,d} \,=\,0$ if either $d\,=\,0$ or $n$ is odd, and
$A_{4n,d} \,=\, {\mathbb Z}/2{\mathbb Z}$ otherwise.
The Picard group of $\mathcal{M}_C(\Omega_{4n}^\delta)$ is
generated by the descent of the Pffafian bundle if $n$ is even 
with $d\,\neq\, 0$, and in other cases
${\rm Pic}(\mathcal{M}_C(\Omega_{4n}^\delta))$ is
generated by the descent of the square of the Pffafian bundle.
\end{enumerate}
\end{proposition}

The two exceptional groups $E_6$ and $E_7$ have nontrivial center.
Using Theorem \ref{direct} and Corollary \ref{corl} the Brauer
group of the moduli space and the moduli stack
for corresponding adjoint type groups can be computed as above.


\end{document}